\documentclass[12pt,article]{amsart}
\usepackage{amsmath}
\usepackage{amssymb}
\newtheorem{lemma}[equation]{Lemma}

\newtheorem{theorem}[equation]{Theorem}

\newcommand{\ddc}{{\mathrm{dd^c}}}
\newcommand{\C}{{\mathbb{C}}}

\renewcommand{\P}{{\mathbb{P}}}

\newcommand{\R}{{\mathbb{R}}}

\newcommand{\rank}{\mathrm{rank}}

\numberwithin{equation}{section}

\begin{document}
\title[On generalized Gauss maps of minimal surfaces sharing hypersurfaces]{On generalized Gauss maps of minimal surfaces sharing hypersurfaces in a projective variety} 
\author{Si Duc Quang}
\address{(Si Duc Quang) Department of Mathematics, Hanoi National University of Education, 136-Xuan Thuy, Cau Giay, Hanoi, Vietnam}
\email{quangsd@hnue.edu.vn}
\author{Do Thi Thuy Hang}
\address{(Do Thi Thuy Hang) Department of Mathematics, Thang Long University, Nghiem Xuan Yem, Hoang Mai, Hanoi, Vietnam}
\email{hangdtt@thanglong.edu.vn}
\maketitle
\begin{abstract}
In this article, we study the uniqueness problem for the generalized gauss maps of minimal surfaces (with the same base) immersed in $\R^{n+1}$ which have the same inverse image of some hypersurfaces in a projective subvariety $V\subset\P^n(\C)$. As we know, this is the first time the unicity of generalized gauss maps on minimal surfaces sharing hypersurfaces in a projective varieties is studied. Our results generalize and improve the previous results in this field.
\end{abstract}
\def\thefootnote{\empty}

\section{Introduction and Main results} 

Let $x_1:S_1\rightarrow \R^{n+1}$ and $x_2:S_2\rightarrow\R^{n+1}$ be two oriented non-flat minimal surfaces immersed in $\R^{n+1}$ and
let $G_1:S_1\rightarrow \P^{n}(\C)$ and $G_2:S_2\rightarrow \P^{n}(\C)$ be their generalized Gauss maps. 
Assume that there is a conformal diffeomorphism $\Phi$ of $S_1$ onto $S_2$ and the Gauss map of the minimal surface $x_2\circ\Phi: S_1\rightarrow\P^n(\C)$ is given by $G_2\circ\Phi$. Then $f^1=G_1$, $f^2=G_2\circ\Phi$ are two nonconstant holomorphic maps from $S_1$ into $\P^n(\C)$.
In 1993, Fujimoto obtained the following result.

\vskip0.2cm
\noindent
\textbf{Theorem A} (cf. \cite[Theorem 1.2]{Fu93a}). {\it Under the notation be as above, let $H_1,\ldots,H_q$ be $q$ hyperplanes of $\P^n(\C)$ in general position such that
\begin{itemize}
\item[(a)]  $(f^1)^{-1}(H_j)=(f^2)^{-1}(H_j)$ for every $j,$
\item[(b)] $f^1=f^2$ on $\bigcup_{j=1}^q(f^1)^{-1}(H_j)\setminus K$ for a compact subset $K$ of $S_1$.
\end{itemize}
Then we have necessarily $f^1=f^2$
\begin{itemize}
\item[(1)] if $q>(n+1)^2+\frac{n(n+1)}{2}$ for the case where $S_1$ is complete and has infinite total curvature or
\item[(2)] if $q\ge (n+1)^2+\frac{n(n+1)}{2}$ for the case where $K=\o$ and $S_1$ and $S_2$ are both complete and have finite total curvature.
\end{itemize}}

In 2017, J. Park and M. Ru \cite{RP} considered the case where $f^1$ and $f^2$ are linearly nondegenerate with an addition assumption that $\bigcap_{j=1}^k(f^1)^{-1}(H_{i_j})=\o$ for every $1\le i_1<\cdots<i_k\le q\ (k\ge 2)$.

Recently, in \cite{Q21}, the author initially studied the modified defect relation for the Gauss map of a minimal surface into a projective variety with hypersurfaces in subgeneral position. Motivated by the methods of \cite{Q19,Q21}, in this paper, we will generalize the above mentioned results to the cases where gauss maps into a projective subvariety of $\P^n(\C)$ have the same inverse image for some hypersurfaces in subgeneral position.

In order to state our results, we recall the following. Let $S$ be an open complete Riemann surface in $\R^{n+1}$. Let $f$ be a holomorphic map from $S$ into an $\ell$-dimension projective subvariety $V$ of $\P^n(\C)$ and let $Q$ be a hypersurface in $\P^n(\C)$ of degree $d$. By $\nu_{Q(f)}$ we denote the pull-back of the divisor $Q$ by $f$. Let $F=(f_0,\ldots,f_n)$ be a reduced representation of $f$. Assume that, the hypersurface $Q$ has a defining polynomial, denoted again by the same notation $Q$ (throughout this paper) if there is no confusion, given by
$$ Q(x_0,\ldots,x_n)=\sum_{I\in\mathcal T_d}a_Ix^I,$$ 
where $\mathcal T_d=\{(i_0,\ldots,i_n)\in\mathbb Z^{n+1}_+; i_0+\cdots+i_n=d\}$, $a_I\in\C$ are not all zero for $I\in\mathcal T_d$ and $x^I=x_0^{i_0}\ldots x_n^{i_n}$ for each $i=(i_0,\ldots,i_n)$. We set
$$ Q(F)=\sum_{I\in\mathcal T_d}a_If^I,$$
where $f^I=f_0^{i_0}\ldots f_n^{i_n}$ for each $I\in\mathcal T_d$. Throughout this paper, for each given hypersurface $Q$ we assume that $\|Q\|=(\sum_{I\in\mathcal T_d}|a_I|^2)^{1/2}=1$. 

Denote by $I(V)$ the ideal of homogeneous polynomials in $\C [x_0,...,x_n]$ defining $V$ and by $\C [x_0,...,x_n]_d$ the vector space of all homogeneous polynomials in $\C [x_0,...,x_n]$ of degree $d$ including the zero polynomial. Define 
$$I_d(V):=\dfrac{\C [x_0,...,x_n]_d}{I(V)\cap \C [x_0,...,x_n]_d}\text{ and }H_V(d):=\dim I_d(V).$$
Denote by $[D]$ the equivalent class  in $I_d(V)$ of the element $D\in \C [x_0,...,x_n]_d$. 

For the variety $V$ of $\P^n(\C)$ such that $f(S)\subset V$, we say that $f$ is nondegenerate over $I_d(V)$ if there is no $[Q]\in I_d(V)\setminus \{0\}$ such that $Q(F)\equiv 0$.

Let $Q_1,...,Q_q\ (q\ge N+1)$ be $q$ hypersurfaces in $\P^n(\C)$. The hypersurfaces $Q_1,...,Q_q$ are said to be in $N$-subgeneral position with respect to $V$ if 
$$ V\cap \left(\bigcap_{j=1}^{N+1}Q_{i_j}\right)=\varnothing \ \forall\ 1\le i_1<\cdots <i_{N+1}\le q.$$
Our first main result is stated as follows.
\begin{theorem}\label{1.1} 
Let $V$ be an $\ell$-dimension projective subvariety of $\P^n(\C)$. Let $S_1,S_2$ be non-flat minimal surfaces immersed in $\R^{n+1}$ with the Gauss maps $G_1,G_2$ into $V$, respectively. Assume that there are conformal diffeomorphisms $\Phi_i$ of $S_1$ onto $S_2$. Let $f^1=G_1,f^2=G_2\circ\Phi$.  Let $Q_1,\ldots,Q_q$ be $q$ hypersurfaces of $\P^n(\C)$ in $N-$subgeneral position with respect to $V$, $d=\mathrm{lcm}(\deg Q_1,\ldots,\deg Q_q)$ and let $k$ be a positive integer such that: 
\begin{itemize}
\item[(a)]  $(f^1)^{-1}(Q_j)=(f^2)^{-1}(Q_j)$ for every $j\in\{1,\ldots,q\},$
\item[(b)] $\bigcap_{j=0}^k(f^1)^{-1}(Q_{i_j})=\o$ for every $1\le i_0<\cdots<i_k\le q$,
\item[(c)] $f^1=f^2$ on $\bigcup_{j=1}^q(f^1)^{-1}(Q_j)$.
\end{itemize}
Suppose that $f^1$ is linear nondegenerate over $I_d(V)$. If $S^1$ is complete and
$$q>\frac{2N-\ell+1}{\ell+1}\left (M+1+\dfrac{2(\sigma_{M}-\sigma_{M-\min\{k,\ell\}})}{d}+\frac{M(M+1)}{2d}\right)$$
where $M=H_d(V)-1$, $\sigma_p=\frac{p(p+1)}{2}$ for every $p\ge 0$ and $\sigma_p=0$ for every $p\le 0$,
then $f^1\equiv f^2$.
\end{theorem}

\textit{Remark 1:} If $V$ is the smallest linear subspace of $\P^n(\C)$ containing $f^1(S)$ and $Q_1,\ldots,Q_q$ are hyperplanes of $\P^n(\C)$ in general position, then $V=\P^\ell(\C)\subset\P^n(\C)$, $d=1$, $N=n,M=\ell$. Therefore, from Theorem \ref{1.1}, $f^1=f^2$ if
$$ q>\frac{2n-\ell+1}{\ell+1}\left (\ell+1+\frac{3\ell(\ell+1)}{2}\right)=\frac{(2n-\ell+1)(3\ell+2)}{2}.$$
This condition is always fulfilled if $q>\frac{(n+1)(3n+2)}{2}=(n+1)^2+\frac{n(n+1)}{2}$ (without any condition on $f^1(S)$).
Then this theorem give an improvement for Theorem A(1).

\begin{theorem}\label{1.2} 
Let $V$ be an $\ell-$dimension projective subvariety of $\P^n(\C)$. Let $S_1,S_2$ be non-flat minimal surfaces in $\R^{n+1}$ with the Gauss maps $G_1,G_2$ into $V$, respectively. Assume that there are conformal diffeomorphisms $\Phi$ of $S_1$ onto $S_2$. Let $f^1=G_1,f^2=G_2\circ\Phi$.  Let $Q_1,\ldots,Q_q$ be $q$ hypersurfaces  (not containing $V$) of $\P^n(\C)$ in $N-$subgeneral position with respect to $V$, $d=\mathrm{lcm}(\deg Q_1,\ldots,\deg Q_q)$  and let $k$ be a positive integer such that: 
\begin{itemize}
\item[(a)]  $(f^1)^{-1}(Q_j)=(f^2)^{-1}(Q_j)$ for every $j\in\{1,\ldots,q\},$
\item[(b)] $\bigcap_{j=0}^k(f^1)^{-1}(Q_{i_j})=\o$ for every $1\le i_0<\cdots<i_k\le q$,
\item[(c)] $f^1=f^2$ on $\bigcup_{j=1}^q(f^1)^{-1}(Q_j)$.
\end{itemize}
If $f^1$ is nondegenerate over $I_d(V)$, $S^1$ is complete, $q\ge 2Mk+2k$ and
$$q>\frac{2N-\ell+1}{\ell+1}\left (M+1+\frac{2Mkq}{(q+2(M-1)k)d}+\frac{M(M+1)}{2d}\right)$$
then there is $\left[\frac{q}{2}\right]$ indices $i_1,\ldots,i_{[q/2]}\in\{1,\ldots,q\}$ such that
$$ \dfrac{Q_{i_1}(F^1)}{Q_{i_1}(F^2)}=\cdots=\dfrac{Q_{i_{[q/2]}}(F^1)}{Q_{i_{[q/2]}}(F^2)}$$
for any two representations $F^1,F^2$ of $f^1,f^1$, respectively.
\end{theorem} 
\textit{Remark 2:} In the above theorem, suppose that $V=\P^n(\C)$, $Q_1,\ldots,Q_q$ are hyperplanes of $\P^n(\C)$ in general position. Then $d=1$, $M=N=\ell=n$. Therefore, from the above theorem, $f^1=f^2$ if $q\ge 2nk+2k$ and
$$q>n+1+\frac{2nkq}{q+2nk-2k}+\frac{n(n+1)}{2}.$$
Therefore, this result implies the previous result of J. Park and M. Ru in \cite{RP}.

\section{Main lemmas}

Let $V$ be $\ell$-dimension subvariety of $\P^n(\C)$. Let $d$ be a positive integer.  Throughout this section and Section 3, we fix a $\C$-ordered basis $\mathcal V=([v_0],\ldots,[v_M])$ of $I_d(V)$, where $v_i\in H_d$ and $M=H_V(d)-1$. 

 Let $S$ be an open Riemann surface and let $z$ is a conformal coordinate. Let $f$ be a holomorphic map of $S$ into $V$, which is nondegenerate over $I_d(V)$.  
Suppose that $F=(f_0,\ldots,f_n)$ is a reduced representation of $f$. 
We set
$$ F=(v_0(F),\ldots,v_M(F))$$
and 
$$F_{p}:= F^{(0)} \wedge F^{(1)} \wedge \cdots \wedge F^{(p)} : S\rightarrow \bigwedge_{p+1}\C^{M+1}$$
for $0\le p\le M$, where 
\begin{itemize}
\item $F^{(0)}:=F=(v_0(F),\ldots,v_M(F))$,
\item $F^{(l)}=F^{(l)}:=\left (v_0(F)^{(l)},\ldots, v_M(F)^{(l)}\right)$ for each $l=0, 1,\ldots , p$,
\item $v_i(F)^{(l)} \ (i =0,\ldots, M)$ is the $l^{th}$- derivatives of $v_i(F)$ taken with respect to $z$.
\end{itemize}
The norm of $F_{p}$ is given by
$$|F_{p}|:=\left (\sum_{0\le i_0<i_1<\cdots<i_p\le M}\left |W(v_{i_0}(F),\ldots,v_{i_p}(F))\right|^2\right)^{1/2}, $$
where 
$$W(v_{i_0}(F),\ldots,v_{i_p}(F)):=\det\left (v_{i_j}(F)^{(l)}\right)_{0\le l,j\le p}.$$ 

Denote by $\langle,\rangle$ the canonical hermitian product on $\bigwedge^{k+1}\C^{M+1}\ (0\le k\le M)$. For two vectors $A\in \bigwedge^{k+1}\C^{M+1}\ (0\le k\le M)$ and $B\in\bigwedge^{p+1}\C^{M+1}\ (0\le p\le k)$, there is one and only one vector $C\in\bigwedge^{k-p}\C^{M+1}$ satisfying 
$$ \langle C,D\rangle=\langle A,B\wedge D\rangle\ \forall D\in \bigwedge^{k-p}\C^{M+1}.$$
The vector $C$ is called the interior product of $A$ and $B$, and denoted by $A\vee B$.

Now, for a hypersurface $Q$ of degree $d$ in $\P^n(\C)$, we have
$$[Q]=\sum_{i=0}^Ma_i[v_i].$$
Hence, we associate $Q$ with the vector $(a_0,\ldots,a_M)\in\C^{M+1}$ and define $F_{p}(Q)=F_{p}\vee H$. Then, we may see that
\begin{align*}
F_{0}(Q)&=a_0v_0(F)+\cdots+a_Mv_M(F)=Q(F),\\ 
|F_{p}(Q)|&=\left (\sum_{0\le i_1<\cdots<i_p\le M}\sum_{l\ne i_1,\ldots,i_p}a_l\left |W(v_l(F),v_{i_1}(F),\ldots,v_{i_p}(F))\right|^2\right)^{1/2}.
\end{align*}
For $0\le p\le M$, the $p^{th}$-contact function of $f$ for $Q$ is defined by
$$\varphi_{p}(Q):=\dfrac{|F_{p}(Q)|^2}{|F_{p}|^2}.$$


\begin{lemma}[{cf. \cite[Lemma 3]{QA}}]\label{lem2.1}
Let $Q_1,...,Q_q$ be $q\ (q>2N-\ell+1)$ hypersurfaces of $\P^n(\C)$ in $N$-subgeneral position with respect to $V$ of the same degree $d$. Then, there are positive rational constants $\omega_i\ (1\le i\le q)$ satisfying the following:

i) $0<\omega_i \le 1\  \forall i\in\{1,...,q\}$,

ii) Setting $\tilde \omega =\max_{j\in Q}\omega_j$, one gets $\sum_{j=1}^{q}\omega_j=\tilde \omega (q-2N+\ell-1)+\ell+1.$

iii) $\dfrac{\ell+1}{2N-\ell+1}\le \tilde\omega\le\dfrac{\ell}{N}.$

iv) For each $R\subset \{1,...,q\}$ with $\sharp R = N+1$, then $\sum_{i\in R}\omega_i\le \ell+1$.

v) Let $E_i\ge 1\ (1\le i \le q)$ be arbitrarily given numbers. For each $R\subset \{1,...,q\}$ with $\sharp R = N+1$,  there is a subset $R^o\subset R$ such that $\sharp R^o=\rank \{[Q_i]\}_{i\in R^o}=\ell+1$ and 
$$\prod_{i\in R}E_i^{\omega_i}\le\prod_{i\in R^o}E_i.$$
\end{lemma}

The following theorem is due to the author in recent works \cite{Q21,Q23,Q24}. 
%
\begin{theorem}[{cf. \cite[Theorem 3.3]{Q21},\cite[Theorem 3.5]{Q23},\cite[Theorem 2.7]{Q24}}]\label{thm2.2}
Let the notations be as above and let $\tilde\omega$ be the constant defined in the Lemma \ref{lem2.1} with respect to the hypersurfaces $Q_1,\ldots,Q_q$. Then, for every $\epsilon>0$, there exist a positive number $\delta\ (>1)$ and $C$, depending only on $\epsilon$ and $Q_j$ such that
\begin{align*}
\ddc&\log\dfrac{\prod_{p=0}^{M-1}|F_{p}|^{2\epsilon}}{\prod_{1\le j\le q,0\le p\le M-1}\log^{2\omega_j}\left(\delta/\varphi_{p}(Q_j)\right)}\\
&\ge C\left (\dfrac{|F_{0}|^{2\left (\tilde\omega(q-(2N-k+1))-M+k\right)}|F_{M}|^2}{\prod_{j=1}^q(|F_{0}(Q_j)|^2\prod_{p=0}^{M-1}\log^2(\delta/\varphi_{p}(Q_j)))^{\omega_j}}\right)^{\frac{2}{M(M+1)}}\ddc|z|^2.
\end{align*}
\end{theorem}

\begin{theorem}[{cf. \cite[Proposition 2.5.7]{Fu93}}]\label{thm2.3}
Set $\tau_m=\sum_{p=1}^m\sigma_m$ for each integer $m$. We have
$$ \ddc\log(|F_{0}|^2\cdots |F_{M-1}|^2)\ge\dfrac{\tau_M}{\sigma_M}\left(\dfrac{|F_{0}|^2\cdots |F_{M}|^2}{|F_{0}|^{2\sigma_{M+1}}}\right)^{1/\tau_M}\ddc|z|^2. $$
\end{theorem}

\begin{theorem}\label{thm2.4}
Let the notations be as above and let the assumption be as in Lemma \ref{lem2.1}, we have
$$\nu_{F^1_M}\ge \sum_{j=1}^q\omega_j\nu_{Q_j(F)}-(\sigma_M-\sigma_{M-\min\{k,\ell\}})\nu^{[1]}_{\prod_{j=1}^qQ_j(F)}.$$

\end{theorem}

\begin{proof}
For a point $a\in\bigcup_{j=1}^q(f^1)^{-1}(Q_j)$, since $\{Q_j\}_{j=1}^q$ is in $N$-subgeneral position with respect to $V$, there are at most $N$ indices $j$ such that $Q_j(F^1)(a)=0$. Then, there is a subset $R\subset\{1,\ldots,q\}$ with $\sharp R=N+1$ such that $Q_j(F^1)(a)\ne 0\ \forall j\not\in R$. Applying Lemma \ref{lem2.1}, there exists a subset $R^0\subset R$ with $\sharp R^o=\ell+1$ such that $\rank_{\C}\{[Q_j];j\in R^o\}=\ell+1$ and
$$ \sum_{j=1}^q\omega_j\nu_{Q_j(F)}(a)=\sum_{j\in R}\omega_j\nu_{Q_j(F^1)}(a)\le\sum_{j\in R^o}\nu_{Q_j(F^1)}(a).$$
We set $k'=\min\{k,\ell\}$. Since there are at most $k'$ indices $j\in R^o$ such that $Q_j(F^1)(a)=0$, we also may assume further that $R^o=\{1,\ldots,\ell+1\}$, $Q_j(F^1)(a)\ne 0$ for all $j> k',j\in R^o$. By the basis property of the wronskian, we have
$$\nu_{F^1_M}(a)\ge \min_\alpha\left\{\sum_{j=1}^{k'}\max\{0,\nu_{Q_j(f^{1})}(a)-(M-\alpha(j))\}\right\}\ge \sum_{j=1}^{k'}\nu_{Q_j(f^{1})}(a)-(\sigma_M-\sigma_{M-k'}),$$
where the minimum is taken over all bijections $\alpha:\{1,\ldots,k'\}\rightarrow\{0,\ldots,k'-1\}$. Thus
$$\nu_{F^1_M}\ge \sum_{j=1}^q\omega_j\nu_{Q_j(F)}-(\sigma_M-\sigma_{M-\min\{k,\ell\}})\nu^{[1]}_{\prod_{j=1}^qQ_j(F)}.$$
The theorem is proved.
\end{proof}

\begin{lemma}[{Generalized Schwarz's Lemma \cite{A38}}]\label{lem2.5}  
Let $v$ be a non-negative real-valued continuous subharmonic function on $\Delta (R)=\{z\in\C; |z|<R\}$. If $v$ satisfies the inequality $\Delta\log v\ge v^2$ in the sense of distribution, then
$$v(z) \le \dfrac{2R}{R^2-|z|^2}.$$
\end{lemma}

\section{Holomorphic curves from complex discs into projective varieties}

\begin{lemma}\label{lem3.1} 
Let $V$ be an $\ell$-dimension projective subvariety of $\P^n(\C)$. Let $Q_1,\ldots,Q_q$ be $q$ hypersurfaces of $\P^n(\C)$ in $N-$subgeneral position with respect to $V$ and let $d$ be the least common multiple of $\deg Q_1,\ldots,\deg Q_q$. Let $f^1,\ldots,f^m$ be $m$ holomorphic maps from $\Delta (R)$ into $V\ (1\le m\le n+1)$, which are nondegenerate over $I_d(V)$. Assume that there exists a holomorphic function $h$ on $\Delta (R)$ satisfying
$$ \lambda\nu_h+\sum_{i=1}^m\nu_{F^i_M}\ge \sum_{i=1}^m\sum_{j=1}^q\omega_j\nu_{Q_j(F^i)}\text{ and }|h|\le\prod_{i=1}^m\|F^i\|^\rho,$$
where $F^i=(F^i_0,\ldots,F^i_n)$ is a reduced representation of $f^i\ (1\le i\le m)$, $\lambda$ and $\rho$ are non-negative numbers. Then for an arbitrarily given $\epsilon$ satisfying 
$$\gamma=\sum_{j=1}^q\omega_j -M-1-\frac{\lambda\rho}{d}>\epsilon\left(\sigma_{M+1}+\frac{\rho}{d}\right).$$ 
the pseudo-metric $d\tau^2=\eta^{2/m}|dz|^2$, where
$$ \eta=\left( |h|^{\lambda+\epsilon}\prod_{i=1}^m\dfrac{|F^i_{0}|^{\gamma-\epsilon(\sigma_{M+1}+\frac{\rho}{d})}|F^i_{M}|\prod_{p=0}^{M}|F^i_{p}|^{\epsilon}}{\prod_{j=1}^q\left (|Q_j(F^i)|\cdot\prod_{p=1}^{M-1}\log (\delta^i/\varphi^i_{p}(Q_j))\right)^{\omega_j}}\right )^{\frac{1}{\sigma_M+\epsilon\tau_M}}$$
and $\delta^i$ is the number satisfying the conclusion of Theorem \ref{thm2.2} with respect to the map $f^i$, is continuous and has strictly negative curvature.
\end{lemma}
Here and throughout this paper, $F^i_p$ and $\varphi^i_{p}$ are defined with respect to the map $f^i$. For simplicity, we sometimes write $\prod_{i,j}$ and $\prod_{j,p}$ for $\prod_{i=1}^m\prod_{j=1}^{q}$ and $\prod_{j=1}^q\prod_{p=1}^{M-1}$, respectively.
\begin{proof}
We see that the function $\eta$ is continuous at every point $z$ with $\prod_{i,j}Q_j(F^i)(z)\ne 0$. For a point $z_0\in \Delta(R)$ such that $\prod_{i,j}Q_j(F^i)(z_0)= 0$, we have
$$\nu_{\eta}(z_0)\ge \frac{1}{\sigma_M+\epsilon\tau_M}\left (\lambda\nu_{h}(z_0)+\sum_{i=1}^m\nu_{F^i_{M}}(z_0)-\sum_{i=1}^m\sum_{j=1}^q\omega_j\nu_{Q_j(F^i)}(z_0)\right)\ge 0.$$
This implies that $d\tau^2$ is a continuous pseudo-metric on $\Delta(R)$. 

We now prove that $d\tau^2$ has strictly negative curvature on $\Delta (R)$. Again, we have
$$\sum_{i=1}^m\ddc\log\dfrac{|F^i_{M}|^{1+\epsilon}}{\prod_{j=1}^q|Q_j(F)|^{\omega_j}}+(\lambda+\epsilon) \ddc\log |h|\ge 0.$$
Let $\Omega$ be the Fubini-Study form of $\P^n(\C)$ and denote by $\Omega_{f^i}$ the pull-back of $\Omega$ by the map $f^i\ (1\le i\le m)$. By Theorems \ref{thm2.2} and \ref{thm2.3}, we have
\begin{align}\label{3.2}
\begin{split}
&\ddc\log\eta^{1/m}\ge\dfrac{\gamma-\epsilon(\sigma_{M+1}+\frac{\rho}{d})}{m(\sigma_M+\epsilon\tau_M)}d\sum_{i=1}^m\Omega_{f^i}\\
&+\dfrac{\epsilon}{4m(\sigma_M+\epsilon\tau_M)}\sum_{i=1}^m\ddc\log\left(|F^i_{0}|^2\cdots|F^i_{M-1}|^2\right)\\
& +\dfrac{1}{2m(\sigma_M+\epsilon\tau_M)}\sum_{i=1}^m\ddc\log\dfrac{\prod_{p=0}^{M-1}|F^i_{p}|^{2(\frac{\epsilon}{2})}}{\prod_{p=0}^{M-1}\log^{2\omega_j}(\delta^i/\varphi^i_{p}(Q_j))}\\
&\ge\dfrac{\epsilon\tau_M}{4m\sigma_M(\sigma_M+\epsilon\tau_M)}\sum_{i=1}^m\left(\dfrac{|F^i_{0}|^2\cdots |F^i_{M}|^2}{|F^i_{0}|^{2\sigma_{M+1}}}\right)^{\frac{1}{\tau_M}}\ddc|z|^2\\
&+C_0\sum_{i=1}^m\left (\dfrac{|F^i_{0}|^{2\left(\tilde\omega (q-2N+k-1)-M+k\right)}|F^i_{M}|^2}{\prod_{j=1}^q(|Q_j(F^i)|^2\prod_{p=0}^{M-1}\log^2(\delta^i/\varphi^i_{p}(Q_j)))^{\omega_j}}\right)^{\frac{1}{\sigma_M}}\ddc|z|^2\\
&\ge\dfrac{\epsilon\tau_M}{4\sigma_M(\sigma_M+\epsilon\tau_M)}\left(\prod_{i=1}^m\dfrac{|F^i_{0}|^2\cdots |F^i_{M}|^2}{|F^i_{0}|^{2\sigma_{M+1}}}\right)^{\frac{1}{m\tau_M}}\ddc|z|^2\\
&+mC_0\left (\prod_{i=1}^m\dfrac{|F^i_{0}|^{2\left(\tilde\omega (q-2N+k-1)-M+k\right)}|F^i_{M}|^2}{\prod_{j=1}^q(|Q_j(F^i)|^2\prod_{p=0}^{M-1}\log^2(\delta^i/\varphi^i_{p}(Q_j)))^{\omega_j}}\right)^{\frac{1}{m\sigma_M}}\ddc|z|^2\\
&\ge C_1\left (\prod_{i=1}^m\dfrac{|F^i_{0}|^{\tilde\omega (q-2N+k-1)-M+k-\epsilon\sigma_{M+1}}|F^i_{M}|\prod_{p=0}^M|F^i_{p}|^\epsilon}{\prod_{j=1}^q\left(|Q_j(F^i)|\prod_{p=0}^{M-1}\log(\delta^i/\varphi^i_{p}(Q_j))\right)^{\omega_j}}\right)^{\frac{2}{m(\sigma_M+\epsilon\tau_M)}}\ddc|z|^2
\end{split}
\end{align}
for some positive constant $C_0,C_1$, where the last inequality comes from H\"{o}lder's inequality.
On the other hand, we have $|h|\le \prod_{i=1}^m\|F^i\|^{\rho}\le \prod_{i=1}^m|F^i_0|^{\frac{\rho}{d}}$ and
$$\prod_{i=1}^m|F^i_{0}|^{\tilde\omega (q-2N+k-1)-M+k-\epsilon\sigma_{M+1}}\ge |h|^{\lambda+\epsilon}\prod_{i=1}^m|F^i_{0}|^{\gamma-\epsilon(\sigma_{M+1}+\frac{\rho}{d})}.$$
This implies that $\Delta \log\eta^{2/m}\ge C_2\eta^{2/m}$ for some positive constant $C_2$. Therefore, $d\tau^2$ has strictly negative curvature.
\end{proof}

\begin{lemma}\label{lem3.3}
Let the notations and the assumption as in Lemma \ref{lem3.1}. Then for an arbitrarily given $\epsilon$ satisfying 
$$\gamma=\sum_{j=1}^q\omega_j -M-1-\frac{\lambda\rho}{d}>\epsilon(\sigma_{M+1}+\frac{\rho}{d}),$$ 
there exists a positive constant $C$, depending only on $\epsilon,Q_j\ (1\le j\le q)$, such that
$$\left (|h|^{\lambda+\epsilon}\prod_{i=1}^m\dfrac{|F^i_{0}|^{\gamma-\epsilon(\sigma_{M+1}+\frac{\rho}{d})}|F^i_{M}|^{1+\epsilon}\prod_{j,p}|F^i_{p}(Q_j)|^{\epsilon/q}}{\prod_{j=1}^q|Q_j(F^i)|^{\omega_j}}\right)^{1/m}\le C\left(\dfrac{2R}{R^2-|z|^2}\right)^{\sigma_M+\epsilon\tau_M}.$$
\end{lemma}
\begin{proof}
As in the proof of Lemma \ref{lem3.1}, we have
$$\ddc\log\eta^{1/m}\le C_2\eta^{2/m}\ddc|z|^2.$$
According to Lemma \ref{lem2.5}, this implies that
$$ \eta^{1/m}\le C_3\dfrac{2R}{R^2-|z|^2},$$
for some positive constant $C_3$. Then we have
$$ \left(|h|^{\lambda+\epsilon}\prod_{i=1}^m\dfrac{|F^i_{0}|^{\gamma-\epsilon(\sigma_{M+1}+\frac{\rho}{d})}|F^i_{M}|\prod_{p=0}^{M}|F^i_{p}|^{\epsilon}}{\prod_{j=1}^q\left (|Q_j(F^i)|\cdot\prod_{p=1}^{M-1}\log (\delta^i/\varphi^i_{p}(Q_j))\right)^{\omega_j}}\right )^{\frac{1}{m(\sigma_M+\epsilon\tau_M)}}\le C_3\dfrac{2R}{R^2-|z|^2}.$$
It follows that
$$ \left( |h|^{\lambda+\epsilon}\prod_{i=1}^m\dfrac{|F^i_{0}|^{\gamma-\epsilon(\sigma_{M+1}+\frac{\rho}{d})}|F^i_{M}|^{1+\epsilon}\prod^i_{j,p}|F^i_{p}(Q_j)|^{\frac{\epsilon}{q}}}{\prod_{j=1}^q\left(|Q_j(F^i)|\prod_{p=0}^{M-1}(\varphi^i_{p}(Q_j))^{\frac{\epsilon}{2q}}\log (\delta^i/\varphi^i_{p}(Q_j))\right)^{\omega_j}}\right )^{\frac{1}{m(\sigma_M+\epsilon\tau_M)}}\le C_3\dfrac{2R}{R^2-|z|^2}.$$
Note that the function $x^{\frac{\epsilon}{q}}\log^{\omega}\left (\dfrac{\delta}{x^2}\right)\ (\omega>0,0<x\le 1)$ is bounded. Then we have
$$ \left(|h|^{\lambda+\epsilon}\prod_{i=1}^m\dfrac{|F^i_{0}|^{\gamma-\epsilon(\sigma_{M+1}+\frac{\rho}{d})}|F^i_{M}|^{1+\epsilon}\prod_{j,p}|F^i_{p}(Q_j)|^{\epsilon/q}}{\prod_{j=1}^q|Q_j(F^i)|^{\omega_j}}\right )^{\frac{1}{m(\sigma_M+\epsilon\tau_M)}}\le C_4\dfrac{2R}{R^2-|z|^2},$$
for a positive constant $C_4$. The lemma is proved.
\end{proof}

\begin{lemma}[{cf. \cite[Lemma 1.6.7]{Fu93}}]\label{lem3.4}
Let $d\sigma^2$ be a conformal flat metric on an open Riemann surface $S$. Then for every point $p\in S$, there is a holomorphic and locally biholomorphic map $\Phi$ of a disk $\Delta (R_0)$ onto an open neighborhood of $p$ with $\Phi(0)=p$ such that $\Phi$ is a local isometric, namely the pull-back $\Phi^*(d\sigma^2)$ is equal to the standard (flat) metric on $\Delta(R_0)$, and for some point $a_0$ with $|a_0|=1$, the curve $\Phi (\overline{0,R_0a_0})$ is divergent in $S$ (i.e., for any compact set $K\subset S$, there exists an $s_0<R_0$ such that $\Phi (\overline{0,s_0a_0})$ does not intersect $K$).
\end{lemma}

\begin{theorem}\label{thm3.5}
Let $S$ be an open Riemann surface and $V$ be an $\ell$-dimension projective subvariety of $\P^n(\C)$. Let $f^1,\ldots,f^m$ be $m$ holomorphic curves from $S$ into $V\ (1\le m\le n)$. Let $Q_1,\ldots,Q_q$ be $q$ hypersurfaces of $\P^n(\C)$ in $N-$subgeneral position with respect to $V$ and $d=\mathrm{lcm}(\deg Q_1,\ldots,\deg Q_q)$. Assume that each $f_i$ is nondegenerate over $I_d(V)$, there exists a holomorphic function $h$ on $S$ satisfying
$$ \lambda\nu_h+\sum_{i=1}^m\nu_{F^i_M}\ge \sum_{i=1}^m\sum_{j=1}^q\omega_j\nu_{Q_j(F^i)}\text{ and }|h|\le\prod_{i=1}^m\|F^i\|^\rho,$$
where $F^i=(F^i_0,\ldots,F^i_n)$ is a reduced representation of $f^i\ (1\le i\le m)$ and the metric
$$ ds^2=2|\xi|^{2/m}\cdot \left(\prod_{i=1}^m\|F^i\|\right)^{2/m}|dz^2|, $$ 
where $\xi$ is a nowhere zero holomorphic function, is complete on $S$. Then we have
$$q\le \frac{2N-\ell+1}{\ell+1}\left (M+1+\frac{\lambda\rho}{d}+\frac{M(M+1)}{2d}\right).$$
\end{theorem}

\begin{proof}
If there are some hypersurfaces $Q_j$ such that $V\subset Q_j$, for instance they all are $Q_{q-r+1},\ldots,Q_q\ (0\le r\le N-\ell+1)$, then by setting $N'=N-r,q'=q-r$ we have
\begin{align*}
&\frac{2N-\ell+1}{\ell+1}\left (M+1+\frac{\lambda\rho}{d}+\frac{M(M+1)}{2d}\right)-q\\
&\ \ \ \ \ \ge \frac{2N'-\ell+1}{\ell+1}\left (M+1+\frac{\lambda\rho}{d}+\frac{M(M+1)}{2d}\right)-q'
\end{align*}
and $Q_{1},\ldots,Q_{q-r}$ are in $N'$-subgeneral position with respect to $V$. Then, without loss of generality, we may assume that $V\not\subset Q_j$ for all $j=1,\ldots,q$.

We fix a $\C$-ordered basis $\mathcal V=([v_0],\ldots,[v_M])$ of $I_d(V)$ as in the Section 3. By replacing $Q_i$ with $Q_i^{d/\deg Q_i}\ (1\le i\le q)$ if necessary, we may assume that all $Q_i\ (1\le i\le q)$ are of the same degree $d$. Suppose that
$$ [Q_j]=\sum_{i=0}^Ma_{ji}[v_i],$$
where $\sum_{i=0}^M|a_{ji}|^2=1$.

Since $f^i\ (1\le i\le m)$ is nondegenerate over $I_d(V)$, the contact functions satisfy
$$ F^i_{p}(Q_j)\not\equiv 0, \forall 1\le j\le q,0\le p\le M.$$
Then, for each $j,p\ (1\le j\le q,0\le p\le M)$, we may choose $i_1,\ldots,i_p$ with $0\le i_1<\cdots<i_p\le M$ such that
$$ \psi(F^i)_{jp}=\sum_{s\ne i_1,\ldots,i_p}a_{js}W(v_s(F^i),v_{i_1}(F^i),\ldots,v_{i_p}(F^i))\not\equiv 0.$$
We note that $\psi(F^i)_{j0}=F^i_{0}(Q_j)=Q_j(F^i)$ and $\psi(F^i)_{jM}=F^i_{M}$.

Suppose contrarily that
$$q>\frac{2N-\ell+1}{\ell+1}\left (M+1+\frac{\lambda\rho}{d}+\frac{M(M+1)}{2d}\right).$$
From Theorem \ref{lem2.1}, we have
$$ (q-2N+\ell-1)\tilde\omega=\sum_{j=1}^q\omega_j-\ell-1;\ \tilde\omega\ge\omega_j>0 \text{ and }\tilde\omega\ge\dfrac{\ell+1}{2N-\ell+1}.$$
Therefore,
\begin{align}\label{new2}
\begin{split}
\sum_{j=1}^q\omega_j-M-1-\frac{\lambda\rho}{d}&\ge\tilde\omega(q-2N+\ell-1)-M+\ell-\frac{\lambda\rho}{d}\\ 
&\ge\dfrac{\ell+1}{2N-\ell+1}(q-2N+\ell-1)-M+\ell-\frac{\lambda\rho}{d}\\
&= \dfrac{\ell+1}{2N-\ell+1}\left(q-\dfrac{(2N+\ell-1)(M+1+\frac{\lambda\rho}{d})}{\ell+1}\right)\\
&>\dfrac{\ell+1}{2N-\ell+1}\cdot\dfrac{(2N+\ell-1)M(M+1)}{2d(\ell+1)}=\dfrac{\sigma_M}{d}.
\end{split}
\end{align}
Then, we can choose a rational number $\epsilon\ (>0)$ such that
$$ \dfrac{d(\sum_{j=1}^q\omega_j-M-1-\frac{\lambda\rho}{d})-\sigma_M}{d(\sigma_{M+1}+\frac{\rho}{d})+\tau_M}>\epsilon> \dfrac{d(\sum_{j=1}^q\omega_j-M-1-\frac{\lambda\rho}{d})-\sigma_M}{\frac{1}{mq}+d(\sigma_{M+1}+\frac{\rho}{d})+\tau_M}.$$
We define the following numbers
\begin{align*}
\beta&:=d\left(\sum_{j=1}^q\omega_j-M-1-\frac{\lambda\rho}{d}-\epsilon\left(\sigma_{M+1}+\frac{\rho}{d}\right)\right)>\sigma_M+\epsilon\tau_{M},\\ 
\rho&:=\dfrac{1}{\beta}(\sigma_M+\epsilon\tau_M),\\
\rho^*&:=\dfrac{1}{(1-\rho)\beta}=\dfrac{1}{d(\sum_{j=1}^q\omega_j-M-1-\frac{\lambda\rho}{d})-\sigma_M-\epsilon(d\sigma_{M+1}+\rho+\tau_M)}. 
\end{align*}
It is clear that $0<\rho<1$ and $\frac{\epsilon\rho^*}{mq}>1.$

We consider a set
$$ S'=\{a\in S; \psi (F^i)_{jp}(a)\ne 0,h(a)\ne 0\ \forall 1\le i\le m; j=1,\ldots,q; p=0,\ldots,M\}$$
and define a new pseudo-metric on $S'$ as follows
$$ d\tau^2=|\xi|^{\frac{2(1+\beta\rho\rho^*)}{m}}\left (\frac{1}{|h|^{\lambda+\epsilon}}\prod_{i=1}^m\dfrac{\prod_{j=1}^q|Q_j(F^i)|^{\omega_j}}{|F^i_{M}|^{1+\epsilon}\prod_{j,p}|\psi(F^i)_{jp}|^{\frac{\epsilon}{q}}}\right)^{\frac{2\rho^*}{m}}|dz|^2.$$

Since $Q_j(F^i),F^i_{M},\psi(F^i)_{jp}\ (1\le j\le q)$ and $h$ are all holomorphic functions on $S'$, $d\tau^2$ is flat on $S'$. We now show that $d\tau^2$ is complete on $S'$.

Indeed, suppose contrarily that $S'$ is not complete with $d\tau^2$, there is a divergent curve $\gamma: [0,1)\rightarrow S'$ with finite length. Then, as $t\rightarrow 1$ there are only two cases: either $\gamma(t)$ tends to a point $a$ with
$$ (h\prod_{j=1}^{q}\prod_{p=0}^M\psi(F^i)_{jp})(a)=0$$
or else $\gamma (t)$ tends to the boundary of $S$.

For the first case, by Theorem \ref{thm2.4}, we have
\begin{align*}
\nu_{d\tau}(a)&\le -\biggl (\sum_{i=1}^m\nu_{F^i_M}(a)-\sum_{i=1}^m\sum_{j=1}^q\omega_j\nu_{Q_j(F^i)}(a)+\lambda\nu_h(a)\\ 
&\ \ +\bigl(\epsilon\sum_{i=1}^m\nu_{F^i_M}(a)+\epsilon\nu_{h}(a)+\dfrac{\epsilon}{q}\sum_{i=1}^m\sum_{j,p}\nu_{\psi(F^i)_{jp}}(a)\bigl)\biggl)\frac{\rho^*}{m}\\
&\le -\frac{\epsilon\rho^*}{m}\left(\sum_{i=1}^m\nu_{F^i_M}(a)+\nu_{h}(a)\right)-\dfrac{\epsilon\rho^*}{mq}\sum_{i=1}^m\sum_{j,p}\nu_{\psi(F^i)_{jp}}(a)\le -\dfrac{\epsilon\rho^*}{mq}.
\end{align*}
Then, there is a positive constant $C$ such that
$$ |d\tau|\ge\dfrac{C}{|z-a|^{\frac{\epsilon\rho^*}{mq}}}|dz|$$
in a neighborhood of $a$. Then we get
$$ L_{d\tau}(\gamma)=\int_0^1\|\gamma'(t)\|_{\tau}dt=\int_{\gamma}d\tau \ge\int_\gamma \dfrac{C}{|z-a|^{\frac{\epsilon\rho^*}{mq}}}|dz|=+\infty$$
($\gamma (t)$ tends to $a$ as $t\rightarrow 1$). This is a contradiction. Then, the second case must occur, i.e., $\gamma (t)$ tends to the boundary of $S$ as $t\rightarrow 1$.  

Take a disk $\Delta$ (in the metric induced by $d\tau^2$) around $\gamma(0)$. Since $d\tau$ is flat, by Lemma \ref{lem3.4}, $\Delta$ is isometric to an ordinary disk in the plane. Let $\Phi:\Delta(r)=\{|\omega|<r\}\rightarrow\Delta$ be this isometric with $\Phi(0)=\gamma(0)$. Extend $\Phi$ as a local isometric into $S'$ to a the largest disk possible $\Delta(R)=\{|\omega|<R\}$, and denoted again by $\Phi$ this extension (for simplicity, we may consider $\Phi$ as the exponential map). Since $\Phi$ cannot be extended to a larger disk, it must be hold that the image of $\Phi$ goes to the boundary of $S'$. But, this image cannot go to points $z_0$ of $S$ with $h(z_0)\prod_{i=1}^m\left (F^i_M(z_0)\prod_{j,p}\psi(F^i)_{jp}|(z_0)\right )=0$, since we have already shown that $\gamma(0)$ is infinitely far away in the metric $d\tau^2$ with respect to these points. Then the image of $\Phi$ must go to the boundary $S$. Hence, by again Lemma \ref{lem3.4}, there exists a point $w_0$ with $|w_0|= R$ so that $\Gamma=\Phi(\overline{0,w_0})$ is a divergent curve on $S$.

Our goal now is to show that $\Gamma$ has finite length in the original metric $ds^2$ on $S$, contradicting the completeness of $S$. Let $g^i:=f^i\circ\Phi :\Delta(R)\rightarrow V\subset\P^n(\C)$ be a holomorphic map which is nondegenerate over $I_d(V)$. Then $g^i$ have a reduced representation
$$ G^i=(g^i_0,\ldots,g^i_n),$$
where $g^i_j=f^i_j\circ\Phi\ (1\le i\le m, 0\le j\le n).$
Hence, we have:
\begin{align*}
\Phi^*ds^2&=2|\xi\circ\Phi|^{2/m}\prod_{i=1}^m\|F^i\circ\Phi\|^{2/m}|\Phi^*dz|^2\\
&=2|\xi\circ\Phi|^{2/m}\left(\prod_{i=1}^m\|G^i\|^{2/m}\right)\left|\dfrac{d(z\circ\Phi)}{dw}\right|^2|dw|^2,
\end{align*}
\begin{align*}
G^i_M&=(F^i\circ\Phi)_M=F^i_M\circ\Phi\cdot\left(\dfrac{d(z\circ\Phi)}{dw}\right)^{\sigma_M},\\
\psi(G^i)_{jp}&=\psi(F^i\circ\Phi)_{jp}=\psi(F^i)_{jp}\cdot\left(\dfrac{d(z\circ\Phi)}{dw}\right)^{\sigma_p}, (0\le p\le M).
\end{align*}
On the other hand, since $\Phi$ is locally isometric,
\begin{align*}
&|dw|=|\Phi^*d\tau|\\
&=|\xi\circ\Phi|^{\frac{1+\beta\rho\rho^*}{m}}\left (\dfrac{1}{|h\circ\Phi|^{\lambda+\epsilon}}\prod_{i=1}^m\dfrac{\prod_{j}|Q_j(F^i)\circ\Phi|^{\omega_j}}{|F^i_M\circ\Phi|^{1+\epsilon}\prod_{j,p}|\psi(F^i)_{jp}\circ\Phi|^{\epsilon/q}}\right)^{\rho^*/m}\left|\dfrac{d(z\circ\Phi)}{dw}\right|\cdot|dw|\\
&=|\xi\circ\Phi|^{\frac{1+\beta\rho\rho^*}{m}}\left (\dfrac{1}{|h\circ\Phi|^{\lambda+\epsilon}}\prod_{i=1}^m\dfrac{\prod_{j}|Q_j(G^i)|^{\omega_j}}{|G^i_M|^{1+\epsilon}\prod_{j,p}|\psi(G^i)_{jp}|^{\epsilon/q}}\right)^{\rho^*/m}\left|\dfrac{d(z\circ\Phi)}{dw}\right|^{1+\beta\rho\rho^*}\cdot|dw|
\end{align*}
(because $1+\rho^*(\sigma_M+\epsilon\tau_M)=1+\beta\rho\rho^*$).
This implies that
\begin{align*}
\left|\dfrac{d(z\circ\Phi)}{dw}\right|&=|\xi\circ\Phi|^{-\frac{1}{m}}\left (|h\circ\Phi|^{\lambda+\epsilon}\prod_{i=1}^m\dfrac{|G^i_M|^{1+\epsilon}\prod_{j,p}|\psi(G^i)_{jp}|^{\epsilon/q}}{\prod_{j}|Q_j(G^i)|^{\omega_j}}\right)^{\frac{\rho^*}{m(1+\beta\rho\rho^*)}}\\
&\le |\xi\circ\Phi|^{-\frac{1}{m}}\left (|h\circ\Phi|^{\lambda+\epsilon}\prod_{i=1}^m\dfrac{|G^i_M|^{1+\epsilon}\prod_{j,p}|G^i_p(Q_j)|^{\epsilon/q}}{\prod_{j}|Q_j(G^i)|^{\omega_j}}\right)^{\frac{\rho^*}{m(1+\beta\rho\rho^*)}}\\
&=|\xi\circ\Phi|^{-\frac{1}{m}}\left (|h\circ\Phi|^{\lambda+\epsilon}\prod_{i=1}^m\dfrac{|G^i_M|^{1+\epsilon}\prod_{j,p}|G^i_p(Q_j)|^{\epsilon/q}}{\prod_{j}^q|Q_j(G^i)|^{\omega_j}}\right)^{\frac{1}{m\beta}}.
\end{align*}
Hence, we have
\begin{align*}
\Phi^*ds&\le \sqrt{2}\prod_{i=1}^m\|G^i\|^{\frac{1}{m}}\left (|h\circ\Phi|^{\lambda+\epsilon}\prod_{i=1}^m\dfrac{|G^i_M|^{1+\epsilon}\prod_{j,p}|G^i_p(Q_j)|^{\epsilon/q}}{\prod_{j=1}^q|Q_j(G^i)|^{\omega_j}}\right)^{\frac{1}{m\beta}}|dw|\\
&=\sqrt{2}\left (|h\circ\Phi|^{\lambda+\epsilon}\prod_{i=1}^m\dfrac{\|G^i\|^{l}|G^i_M|^{1+\epsilon}\prod_{j,p}|G^i_p(Q_j)|^{\epsilon/q}}{\prod_{j=1}^q|Q_j(G^i)|^{\omega_j}}\right)^{\frac{1}{m\beta}}|dw|\\
&\le C_1\left (|h\circ\Phi|^{\lambda+\epsilon}\prod_{i=1}^m\dfrac{|G^i_0|^{\frac{l}{d}}|G^i_M|^{1+\epsilon}\prod_{j,p}|G^i_p(Q_j)|^{\epsilon/q}}{\prod_{j=1}^q|Q_j(G^i)|^{\omega_j}}\right)^{\frac{1}{m\beta}}|dw|.
\end{align*}
with a positive constant $C_1$. We note that $\frac{\beta}{d}=\sum_{j=1}^q\omega_j-M-1-\frac{\lambda\rho}{d}-\epsilon\left(\sigma_{M+1}+\frac{\rho}{d}\right)$. Then the inequality (\ref{new2}) yields that the conditions of Lemma \ref{lem2.5} are satisfied. Then, by applying Lemma \ref{lem2.5}, we have
$$ \Phi^*ds\le C_2\left (\dfrac{2R}{R^2-|w|^2}\right)^\rho|dw|$$
for some positive constant $C_2$. Also, we have that $0<\rho<1$. Then
$$L_{ds^2}(\Gamma)\le\int_{\Gamma}ds=\int_{\overline{0,w_0}}\Phi^*ds\le C_2\cdot\int_{0}^R\left(\dfrac{2R}{R^2-|w|^2}\right)^{\rho}|dw|<+\infty. $$
This contradicts the assumption of completeness of $S$ with respect to $ds^2$. Thus, $d\tau^2$ is complete on $S'$.

Then, we note that the metric $d\tau^2$ on $S'$ is flat. Then by a theorem of Huber (cf. \cite[Theorem 13, p.61]{Hu61}), the fact that $S'$ has finite total curvature (with respect to $d\tau^2$) implies that $S'$ is finitely connected. This means that there is a compact subregion of $S'$ whose complement is the union of a finite number of doubly-connected regions. Therefore, the functions $|h|\prod_{p=0}^M\prod_{j=1}^q|\psi(G_z)_{jp}|$ must have only a finite number of zeros, and the original surface $S$ is finitely connected. Due to Osserman (cf. \cite[Theorem 2.1]{O63}), each annular ends of $S'$, and hence of $S$, is conformally equivalent to a punctured disk. Thus, the Riemann surface $S$ must be conformally equivalent to a compact surface $\overline S$ punctured at a finite number of points $P_1,\ldots,P_r$. Then,  there are disjoint neighborhoods $U_i$ of $P_i\ (1\le i\le r)$ in $\overline S$ and biholomorphic maps $\phi_i:U_i\rightarrow\Delta$ with $\phi_i(P_i)=0$. Note that, the Poincare-metric on $\Delta^*=\Delta\setminus\{0\}$ is given by $d\sigma_{\Delta^*}^2=\dfrac{4|dw|^2}{|w|^2\log^2|w|^2}$, where $w$ is the complex coordinate on $\Delta$. 

As we known that the universal covering surface of $S$ is biholomorphic to $\C$ or a disk in $\C$. If the universal covering of $S$ is biholomorphic to $\C$ (denote by $\tilde\Phi :\C\rightarrow S$ this universal covering mapping), then from the assumption that
$$ \lambda\nu_h+\sum_{i=1}^m\nu_{F^i_M}\ge \sum_{i=1}^m\sum_{j=1}^q\omega_j\nu_{Q_j(F^i)}\text{ and }|h|\le\prod_{i=1}^m\|F^i\|^\rho,$$
we have
$$\lambda\rho\sum_{i=1}^mT_{f^i\circ\tilde\Phi}\sum_{j=1}^qN(r,\nu_{h\circ\tilde\Phi})\ge \sum_{i=1}^m\left(\sum_{j=1}^q\omega_jN(r,\nu_{Q_j(F^i\circ\tilde\Phi)})-N(r,\nu_{F^i_M\circ\Phi})\right),$$
where $T_f(r)$ is the characteristic function of the mapping $f:\C\to\P^n(\C)$ and $N(r,\nu)$ is the counting function of the divisor $\nu$ on $\C$ (see \cite{QA} for the definitions). Using the second main theorem (Theorem 1.1 in \cite{QA}), we have
$$ \|\ \lambda\rho\sum_{i=1}^mT_{f^i\circ\tilde\Phi}\ge\sum_{i=1}^m(\sum_{j=1}^qN^{[M]}(r,\nu_{Q_j(F^i\circ\tilde\Phi)})\ge \left(q-\frac{(2N-\ell+1)(M+1)}{\ell+1}\right) \sum_{i=1}^mT_{f^i\circ\tilde\Phi}.$$
Here, the symbol ``$\|$'' means the inequalities hold for all $r\in [1,+\infty)$ outside a finite Borel measure set $E$. Letting $r\rightarrow +\infty\ (r\not\in E)$, we get
$$ \lambda\rho\ge q- \frac{(2N-\ell+1)(M+1)}{\ell+1}$$
and arrive at a contradiction.

Then, we only consider the case where the universal covering surface of $S$ is biholomorphic to the unit disk $\Delta$ in $\C$. Let $\tilde\Phi :\Delta\rightarrow S$ be this holomorphic universal covering. Consider the following metric
$$ d\tilde\tau^2=\eta^2|dz|^2, $$
where
$$ \eta=\left (|h|^{\lambda+\epsilon}\prod_{i=1}^m\dfrac{|F^i_{0}|^{\gamma-\epsilon(\sigma_{M+1}+\frac{\rho}{d})}|F^i_{M}|^{1+\epsilon}\prod_{j,p}|F^i_{p}(Q_j)|^{\epsilon/q}}{\prod_{j=1}^q|Q_j(F^i)|^{\omega_j}}\right)^{\frac{1}{m(\sigma_M+\epsilon\tau_M)}}.$$
It is obvious that $d\tilde\tau^2$ is continuous on $S\setminus\bigcup_{j=1}^q (f^i)^{-1}(Q_j)$. Take a point $a$ such that $\prod_{j=1}^qQ_j(F^i)(a)=0$. From the assumption, we have
$$d\tilde\tau(a)\ge\dfrac{1}{m(\sigma_M+\epsilon\tau_M)}\biggl (\lambda\nu_{h}(a)+\sum_{i=1}^m\nu_{F^i_M}(a)-\sum_{i=1}^m\sum_{j=1}^q\omega_j\nu_{Q_j(F^i)}(a)\biggl)\ge 0.$$
Therefore $d\tilde\tau$ is continuous at $a$. This yields that $d\tilde\tau$ is a continuous pseudo-metric on $S$.
Now, from Lemma \ref{lem3.1}, we see that $d\tau^2$ has strictly negative curvature on $S$
Hence, by the decreasing distance property, we have
$$\Phi^*d\tau^2\le d\sigma_\Delta^2\le C_3\cdot(\Phi\circ\phi_i^{-1})^*d\sigma_{\Delta^*}^2\ (1\le i\le r)$$
for some positive constant $C_3$. This implies that
$$\int_{U_i}\Omega_{d\tau^2}\le \int_{\Phi^{-1}(U_i)}\Phi^*\Omega_{\sigma_\Delta^2}\le l_0C_3\int_{\Delta^*}\Omega_{d\sigma_{\Delta^*}^2}<\infty.$$
where $l_0$ is the number of the sheets of the covering $\Phi$. Then, it yields that
$$\int_S\Omega_{d\tau^2}\le \int_{S\setminus\bigcup_{i=1}^rU_i}\Omega_{d\tau^2}+l_0C_3r\int_{\Delta^*}\Omega_{d\sigma_{\Delta^*}^2}<\infty.$$

Now, denote by $ds^2$ the original metric on $S$. Similar as (\ref{3.2}), we have
$$ \ddc\log\eta \ge\dfrac{\gamma-\epsilon(\sigma_{M+1}+\frac{\rho}{d})}{\sigma_M+\epsilon\tau_M}\dfrac{d}{m}\sum_{i=1}^m\Omega_{f_i}.$$ 
Then there is a subharmonic function $v$ such that
\begin{align*}
\lambda^2|dz|^2&=e^{v}|\xi|^{\frac{2}{m}}\left(\prod_{i=1}^m\|F^i\|^{2\frac{\gamma-\epsilon(\sigma_{M+1}+\frac{\rho}{d})}{m(\sigma_M+\epsilon\tau_M)}}\right)|dz|^2\\
&=e^{v}\left(\prod_{i=1}^m\|F^i\|^{2\frac{\gamma-\sigma_M-\epsilon\tau_{M+1}}{m(\sigma_M+\epsilon\tau_M)}}\right)|\xi|^{\frac{2}{m}}\left(\prod_{i=1}^m\|F^i\|^{\frac{2}{m}}\right) |dz|^2\\
&=e^w ds^2
\end{align*}
for a subharmonic function $w$ on $S$. Since $S$ is complete with respect to $ds^2$, applying a result of S. T. Yau \cite{Y76} and L. Karp \cite{K82} we have
$$ \int_S\Omega_{d\tau^2}=\int_Se^w\Omega_{ds^2}=+\infty.$$
This contradiction completes the proof of the theorem.
\end{proof}

\section{Uniqueness theorems for Gauss maps}

In this section, we will prove main theorems of this paper. Firstly, we prove the following.
\begin{lemma}\label{lem4.1}
Let $S$ be an open Riemann surface and $V$ be a $\ell$-dimension projective subvariety of $\P^n(\C)$. Let $Q_1,\ldots,Q_q$ be $q$ hypersurfaces of $\P^n(\C)$ in $N-$subgeneral position with respect to $V$ and $d=\mathrm{lcm}(\deg Q_1,\ldots,\deg Q_q)$. Let $f^1,f^2$ be holomorphic maps from $S$ into $V$ such that 
\begin{itemize}
\item[(1)] $\bigcap_{j=0}^{k}(f^1)^{-1}(Q_{i_j})=\emptyset$ for every $1\le j_0<\cdots<j_{k}\le q,$
\item[(2)] $f^1=f^2$ on $\bigcup_{j=1}^{q}\left ((f^1)^{-1}(Q_{j})\cup (f^2)^{-1}(Q_{j})\right )$.
\end{itemize}
If $f^1$ is nondegenerate over $I_d(V)$, $S$ is complete with a metric $ds^2=|\xi|^2\|F^1\|^2|dz^2|$, where $\xi$ is a non-vanishing holomorphic function, $z$ is a conformal coordinate on $S$, $F^1$ is a reduced presentation of $f^1$, and 
$$ q> \frac{2N-\ell+1}{\ell+1}\left (M+1+\sigma_M-\sigma_{M-\min\{k,\ell\}}+\frac{M(M+1)}{2d}\right)$$
 then $f^2$ is nondegenerate over $I_d(V)$.
%
\end{lemma}

\begin{proof}
Let $F^i=(f^i_0,\ldots,f^i_n)$ be reduce representations of $f^i\ (i=1,2)$. Suppose contrarily that $f^2$ is degenerate over $I_d(V)$. Then there exists a hypersurface $Q$ of degree $d$ such that $V\not\subset Q$, $Q(F^2)\equiv 0$. By the assumption that $f^1$ is nondegenerate over $I_d(V)$, we have $Q(F^1)\not\equiv 0$. Since $f^1=f^2$ on $\bigcup_{i=1}^qQ_i$, we have $Q(F^1)=0$ on $\bigcup_{i=1}^q(f^1)^{-1}Q_i$. Therefore, setting $k'=\min\{k,\ell\}$ and $h=Q(F^1)$, from Theorem \ref{thm2.4} we have 
$$ (\sigma_M-\sigma_{M-k'})\nu_{h}(a)+\nu_{F^1_M}(a)\ge \sum_{j=1}^{k'}\nu_{Q_j(F^{1})}(a)\ge\sum_{j=1}^q\omega_j\nu_{Q_j(F^1)}(a).$$
Also, it is clear that $ |h|\le \|F^1\|^d.$ Applying Theorem \ref{thm3.5}, we have
$$ q\ge \frac{2N-\ell+1}{\ell+1}\left (M+1+\sigma_M-\sigma_{M-k'}+\frac{M(M+1)}{2d}\right).$$
This contradiction completes the proof of  the lemma.
%
\end{proof}

\begin{proof}[Proof of Theorem \ref{1.1}] Without loss of generality, we may assume that $\deg Q_j=d$ for all $1\le j\le q$. We may suppose that $f^{1}(S)\not\subset Q_j$ for all $j\in\{1,\ldots,q\}$ (otherwise $f^1=f^2$). Let $z$ be a conformal coordinate on $S^1$ and $F^i=(f^i_{0},\ldots,f^i_n)$ be the reduce representation of $f^i$ for each $i\in\{1,2\}$. Since $\Phi$ is a conformal diffeomorphism, there exists a non-vanishing holomorphic function $\xi$ such that $ds^2=\|F^1\|^2|dz^2|=|\xi|^2\|F^2\|^2|dz^2|$. We have $ ds^2$ is complete on $S^1$. Also, from the proof of Lemma \ref{lem4.1}, we see that $f^2(S^1)\subset V$ and $f^2$ is nondegenerate over $I_d(V)$.

Suppose contrarily that $f^1\not\equiv f^2$. 
Then there exists $0\le i<j\le n$ such that
$$h:=f^1_if^2_j-f_1^jf^2_i\not\equiv 0.$$
It is clear that $\nu_h\ge \nu^{[1]}_{\prod_{j=1}^{q}Q_j(F^i)}$ for each $i\in\{1,2\}$. 

From Theorem \ref{thm2.4}, we have
$$2(\sigma_{M}-\sigma_{M-\min\{k,\ell\}})\nu_{h}+\sum_{i=1}^2\nu_{F^i_{M}}\ge \sum_{i=1}^2\sum_{j=1}^{q}\nu_{Q_j(F^i)}.$$
Note that $|h|\le \|F^1\|\cdot\|F^2\|$. By applying Theorem \ref{thm3.5}, we have 
$$ q\le \frac{2N-\ell+1}{\ell+1}\left (M+1+\dfrac{2(\sigma_{M}-\sigma_{M-\min\{k,\ell\}})}{d}+\frac{M(M+1)}{2d}\right).$$
This contradiction completes the proof of the theorem.
\end{proof}

\begin{proof}[Proof of Theorem \ref{1.2}]
Let $z$ be a conformal coordinate on $S^1$ and $F^i$ be the reduce representation of $f^i$ for each $i\in\{1,2\}$. Since $\Phi$ is a conformal diffeomorphism, there exists a non-vanishing holomorphic function $\xi$ such that $ds^2=\|F^1\|^2|dz^2|=|\xi|^2\|F^2\|^2|dz^2|=|\xi|\cdot\|F^1\|\cdot \|F^2\| |dz^2|$. Note that, $ds^2$ is complete on $S^1$.

Suppose contrarily that the theorem does not hold. Consider the simple graph $\mathcal G$ with the set of vertices $\{1,\ldots,q\}$ and the set of edges consisting of all pairs $\{i,j\}$ such that $Q_i(F^1)Q_j(F^2)-Q_j(F^1)Q_i(F^2)\not\equiv 0$. The supposition implies that the order of each vertex does not exceed $[q/2]$. Then, by Dirac's theorem, there is a Hamilton cycle $i_1,\ldots,i_{q},i_{q+1}$, where $i_{q+1}=i_1$. We set $u_j:=i_{j+1}$ if $j<q$ and $u_{q}:=i_1$. Then we have
$$h:=\prod_{j=1}^{q}(Q_{i_j}(F^1)Q_{u_j}(F^2)-Q_{u_j}(F^1)Q_{i_j}(F^2))\not\equiv 0.$$
It is clear that $|h|\le \|F^1\|^{dq}\|F^2\|^{dq}$. 

For each point $a\in\bigcup_{j=1}^{q}(f^1)^{-1}(Q_j)$, take a subset $I_1\subset\{1,\ldots,q\}$ such that $\sharp I=N+1$ and $Q_j(F^1)(a)\ne 0$ for every $j\not\in I_1$. Then there is a subset $I_2\subset I_1$ such that $\sharp I_2=\rank_{\C}\{[Q_j];j\in I_2\}=\ell+1$ and 
$$ \sum_{i\in I_2}(\nu_{Q_i(F^1)}(a)+\nu_{Q_i(F^2)}(a))\ge \sum_{i\in I_1}\omega_i(\nu_{Q_i(F^1)}(a)+\nu_{Q_i(F^2)}(a)).$$
Then, there exists a subset $I\subset I_2$ such that $\sharp I=t$ and $Q_j(F^1)(a)\ne 0$ for every $j\in I_2\setminus I$. Hence, we have
$$\sum_{i\in I}(\nu_{Q_i(F^1)}(a)+\nu_{Q_i(F^2)}(a))\ge \sum_{i=1}^q\omega_i(\nu_{Q_i(F^1)}(a)+\nu_{Q_i(F^2)}(a)).$$
Then, we have
\begin{align*}
\nu_h(a)&\ge 2\sum_{i=1}^{t}\min\{\nu_{Q_i(F^1)}(a),\nu_{Q_i(F^2)}(a)\}+(q-2t)\\
&\ge 2\sum_{i=1}^{t}\big(\min\{\nu_{Q_i(F^1)}(a), M\}+\min\{\nu_{Q_i(F^2)}(a),M\}-M\big)+(q-2t)\\
&=2\sum_{i=1}^{t}\big(\min\{\nu_{Q_i(F^1)}(a), M\}+\min\{\nu_{Q_i(F^2)}(a),M\}\big)+(q-2(M+1)t)\\
&\ge\frac{q+2(M-1)t}{2Mt}\sum_{i=1}^{t}\big(\min\{\nu_{Q_i(F^1)}(a), M\}+\min\{\nu_{Q_i(F^2)}(a),M\}\big).
\end{align*}
Also, by usual arguments we have
\begin{align*}
\nu_{F^1_MF^2_M}(a)\ge&\sum_{j=1}^2\sum_{i=1}^{t}\max\{0,\nu_{Q_i(F^j)}(a)-M \}.
\end{align*}
Therefore, we have the following estimate:
\begin{align*}
\frac{2Mt}{q+2(M-1)t}\nu_h(a)&+\nu_{F^1_MF^2_M}(a)\ge \sum_{i=1}^t(\nu_{Q_i(F^1)}(a)+\nu_{Q_i(F^2)}(a))\\
&\ge \sum_{i=1}^q\omega_i(\nu_{Q_i(F^1)}(a)+\nu_{Q_i(F^2)}(a))
\end{align*}
By Theorem \ref{thm3.5}, we have
\begin{align*}
q&\le \frac{2N-\ell+1}{\ell+1}\left (M+1+\frac{2Mtq}{(q+2(M-1)t)d}+\frac{M(M+1)}{2d}\right)\\
&=\frac{2N-\ell+1}{\ell+1}\left (M+1+\frac{2Mkq}{(q+2(M-1)k)d}+\frac{M(M+1)}{2d}\right)
\end{align*}
and arrive at a contradiction. This completes the proof of the theorem. 
\end{proof}

\section*{Disclosure statement}
No potential conflict of interest was reported by the authors.

\end{document}